\documentclass[10pt]{amsart}

\usepackage[active]{srcltx}
\usepackage{mathrsfs}
\usepackage[all]{xy}
\usepackage{calrsfs}
\usepackage{times}
\usepackage[scaled=0.92]{helvet}
\usepackage[usenames]{color}

\usepackage{soul}
\usepackage{todonotes}
\usepackage{cleveref}
\usepackage{amsfonts}
\usepackage{amssymb}

\usepackage[usenames]{color}
\newtheorem{theorem}{Theorem}   
 
\newtheorem{lemma}[theorem]{Lemma}
\newtheorem{proposition}[theorem]{Proposition}
\newtheorem{corollary}[theorem]{Corollary}
\newtheorem{claim}{Claim}

\theoremstyle{definition}
\newtheorem{definition}[theorem]{Definition}
\newtheorem{remark}[theorem]{Remark}

\definecolor{MyDarkBlue}{rgb}{0,0.08,0.60}

\newcommand{\s}{\vspace{0.3cm}}

\newcommand{\red}[1]{{\color{red} #1}}

\begin{document}

\title{Automorphisms of Generalized Fermat curves}

\date{\today}

\author[R. A. Hidalgo]{Rub\'en A. Hidalgo}
\address{Departamento de Matem\'aticas\\Universidad T\'ecnica Federico Santa Mar\'{\i}a\\Valpara\'{\i}so \\Chile}
\email{ruben.hidalgo@mat.utfsm.cl}

\author[A. Kontogeorgis]{Aristides Kontogeorgis}
\address{Department of Mathematics, University of Athens\\
Panepistimioupolis, 15784 Athens, Greece}
\email{kontogar@math.uoa.gr}

\author[M. Leyton-Alvarez]{Maximiliano Leyton-\'{A}lvarez}
\address{Instituto de Matem\'atica y F\'isica, Universidad de Talca\\
Camino Lircay S$\backslash$N, Campus Norte, Talca, Chile}
\email{leyton@inst-mat.utalca.cl}

\author[P. Paramantzoglou]{Panagiotis Paramantzoglou }
\address{Department of Mathematics, University of Athens\\
Panepistimioupolis, 15784 Athens, Greece}
\email{pan\_param@math.uoa.gr}

\thanks{The first and third authors were  partially supported by projects Fondecyt 1110001 and Fondecyt 11121163 respectively while the second author is  supported by the Project ``{\em Thalis, Algebraic modelling of topological and Computational structures}''. The Project ``THALIS'' is implemented under the Operational Project ``Education and Life Long Learning''and is co-funded by the European Union (European Social Fund) and National Resources (ESPA)}

\maketitle

\begin{abstract} 
Let $K$ be an algebraically closed field of characteristic $p \geq 0$. A generalized Fermat curve of type $(k,n)$, where $k,n \geq 2$ are integers (for $p \neq 0$ we also assume that $k$ is relatively prime to $p$), is a non-singular irreducible projective algebraic curve $F_{k,n}$ defined over $K$ admitting a group of automorphisms $H \cong {\mathbb Z}_{k}^{n}$ so that $F_{k,n}/H$ is the projective line with exactly $(n+1)$ cone points, each one of order $k$. Such a group $H$ is called a generalized Fermat group of type $(k,n)$. 
If $(n-1)(k-1)>2$, then $F_{k,n}$ has genus $g_{n,k}>1$ and it is known to be non-hyperelliptic. In this paper, we prove that every generalized Fermat curve of type $(k,n)$ has a unique generalized Fermat group of type $(k,n)$ if $(k-1)(n-1)>2$ (for $p>0$ we also assume that $k-1$ is not a power of $p$). 

Generalized Fermat curves of type $(k,n)$ can be described as a suitable fiber product of $(n-1)$ classical Fermat curves of degree $k$. We prove that, for $(k-1)(n-1)>2$ (for $p>0$ we also assume that $k-1$ is not a power of $p$), each automorphism of such a fiber product curve can be extended to an automorphism of the ambient projective space. In the case that $p>0$ 
and $k-1$ is a power of $p$, we use tools from the theory  of complete projective 
intersections in order to prove that, for $k$ and $n+1$ relatively prime, 
every automorphism of the fiber product curve can also be extended to an automorphism of 
the ambient projective space. 

In this article we also prove that the set of fixed points of the non-trivial elements 
of the generalized Fermat group coincide with  the hyper-osculating points of the fiber product model  under the assumption that the characteristic $p$
is either zero or $p>k^{n-1}$.
\end{abstract}

\section{Introduction} 
In this paper, $K$ will denote an algebraically closed field of characteristic $p \geq 0$. A generalized 
Fermat curve of type $(k,n)$, where $k,n \geq 2$ are integers (and for $p>0$ we also assume that $k$ is relatively 
prime to $p$), is a non-singular irreducible projective algebraic curve $F_{k,n}$ defined over $K$ admitting a 
group of automorphisms $H \cong {\mathbb Z}_{k}^{n}$ so that $F_{k,n}/H$ is the projective line with exactly $(n+1)$ cone points,
each one of order $k$. Such a group $H$ is called a generalized Fermat group of type $(k,n)$. 
If $(n-1)(k-1)>2$, then $F_{k,n}$ has genus $g_{n,k}>1$ (see Section \ref{Sec:2}) and it is
known to be non-hyperelliptic \cite{Har77}. \\

The generalized Fermat curves are objects with a very interesting geometry. 
These curves provide us with a considerable amount of examples, 
and their study could eventually help us to generalize certain important results.
More precisely one of our future objectives is to generalize the work of of Y. Ihara 
\cite{Iha85} on Braid representations of the absolute Galois groups. By Belyi 
theorem he considered covers of the projective line ramified above $\{0,1,\infty\}$ and the 
Fermat curve and its arithmetic emerged naturally. 
If one tries to generalise to the more general $n+1$-ramified covers the generalised 
Fermat curves and their arithmetic emerged in a similar way. This will be the object of another article.

To finish this paragraph, we would like to discuss a second interesting aspect of these curves.  
It is known that the geometry of compact Riemann Surfaces  can be described via 
projective algebraic curves, Fuchsian group, 
Schottky groups, Abelian varieties, etc. However, given one of these 
descriptions, explicitly obtaining the others is a difficult problem, in fact in general it 
is a problem that has not been solved.  The majority of examples of Riemann Surfaces
where we explicitly know the uniformizing Fuchsian group, and the equations of an 
algebraic curve which represents them, are rigid examples, in other words they are not families.  
The generalized Fermat Curves of the type $(n,k)$ over $K=\mathbb{C}$ form a family of algebraic curves of 
complex dimensions $n-2$ in which we explicitly know, for each member of the family, 
a representation as an algebraic curve and the uniformizing Fuchsian group (see \cite{GHL09}).

In the case of arbitrary characteristics, the Generalized Fermat curves can be studied as
Kummer extensions of the rational function field.

We study the full group of automorphisms of generalized Fermat curves and the uniqueness 
of generalized Fermat groups. 
Our main result is Theorem \ref{unicidad-p} which states the uniqueness of generalized 
Fermat groups of type $(k,n)$ if $(k-1)(n-1)>2$ (for $p>0$ we also assume that $k-1$ is not a power of $p$). 

A generalized Fermat curve of type $(k,n)$ can be seen as a complete intersections in a projective 
space defined by the set of equations given in eq. (\ref{defineGFC}) in Section \ref{Sec:2}. Recall that a closed subscheme
$Y$ of $\mathbb{P}^s$ is called a (strict) complete intersection, if the homogeneous ideal in $K[x_1,\ldots,x_{n+1}]$ can be generated by $\mathrm{codim}(Y,\mathbb{P}^s)$ elements. By looking at the defining  equations, we may see the generalized Fermat curves as a suitable fiber product of $(n-1)$ classical Fermat curves of degree $k$. We prove that in such algebraic model the full group of automorphism is a subgroup of the linear group under the assumptions that 
$(n-1)(k-1)>2$ (if $p>0$ we also assume that $k-1$ is not a power of $p$). 

In the case that $p>0$ and $k-1$ is a power of $p$, then we may obtain a similar result under the assumption that $n+1$ is relatively prime to $k$ (Theorem \ref{teolineal}).
The different behaviour in the case $k-1=q=p^h$ 
is an expected phenomenon,  seen also  in the case of the Fermat curves 
$x_1^{q+1}+x_2^{q+1}+x_3^{q+1}=0$, where $q=p^h$, which have $\mathrm{PGU}_{3}(q^{2})$ as automorphism group, see \cite{Leo96}. Essentially this happens since raising to a $p$-power is linear and the Fermat curve in this case behaves like a quadratic form.

Our strategy, in the positive characteristic case, is the following. By a degree argument we show that the group of linear automorphisms is normal in the whole automorphism group. The group of linear automorphisms 
is studied by finding all linear transformations which leave the defining ideal of the curve invariant. For higher dimensional varieties there is an argument proving that every automorphism is linear, based on the fact that the Picard group is free. This argument can not be used in the case of curves, since the Picard groups of curves are known to have torsion. Nevertheless we can use a derivation argument in order to settle some cases.

\bigskip

This paper is organized as follows. In Section \ref{Sec:2} we describe a fiber product
of generalized Fermat curves and introduce the main results of the paper. The most important is Theorem \ref{unicidad-p} which states the uniqueness of the generalized Fermat groups of type $(k,n)$, when $(k-1)(n-1)>2$ (and for $p>0$ the extra assumption that $k-1$ is not a power of $p$). In the fiber product model, under the same hypothesis, we obtain that the full group of automorphisms is linear. The proof of the above is provided in Section \ref{Sec:7}.

In Section \ref{Sec:3} we restrict  our study to zero characteristic or to positive characteristic $p> k^{n-1}$ and prove that the set of 
fixed points of the non-trivial elements of the generalized Fermat group in the fiber product model 
coincide with the set of hyper-osculating points of the fiber product model.

In Section \ref{Sec:4}  we provide the proof of Theorem \ref{teolineal}, concerning  the linearity of the full group of automorphisms in the case when $k-1$ is a power of $p>0$ and $k$ is relatively prime to $p$, under the extra condition that $k$ and $n+1$ are also relatively prime.

\section{Main results}\label{Sec:2}
We use the notation $(a,b)$ to denote the maximum common divisor between the positive integers $a$ and $b$; so $(a,b)=1$ states that $a$ and $b$ are relatively prime integers.

Let $K$ be an algebraically closed field of characteristic $p \geq 0$, let 
$n,k \geq 2$ be integers (if  $p>0$, then we also assume that $(k,p)=1$).

A pair $(F_{k,n},H)$ is called a generalized Fermat pair of the type $(k,n)$ if 
$F_{k,n}$ is a generalized Fermat curve of type $(k,n)$, defined over $K$, and $H \cong {\mathbb Z}_{k}^{n}$ is a generalized Fermat group of type $(k,n)$ of $F_{k,n}$.  The genus of $F_{k,n}$ is

\begin{equation}\label{genero}
g_{(k,n)}=1+ \frac{k^{n-1}}{2} ((n-1)(k-1)-2).
\end{equation}

In particular, $g_{(k,n)}>1$ if and only if $(k-1)(n-1)>2$; in this case the generalized Fermat curve is non-hyperelliptic \cite{Har77}.
If $K={\mathbb C}$, then $F_{k,n}$ defines a closed Riemann surface. Riemann surfaces of this kind  were studied in \cite{GHL09}.

Two generalized Fermat pairs of same type, say $(F_{k,n},H)$ and $(\widehat{F}_{k,n},\widehat{H})$, are called equivalent if there is an isomorphism $\phi:F_{k,n} \to \widehat{F}_{k,n}$ so that $\phi H \phi^{-1}=\widehat{H}$.

\s
\subsection{A fiber product description}
Let us consider a generalized Fermat pair $(F_{k,n},H)$. Let us consider a branched regular covering $\pi:F_{k,n} \to {\mathbb P}^{1}$, whose deck group is $H$. By composing by a suitable M\"obius transformation (that is, an element of ${\rm PSL}_{2}(K)$) at the left of $\pi$, we may assume that the branch values of $\pi$ are given by the points
$$\infty, 0,1, \lambda_{1}, \ldots, \lambda_{n-2},$$
where $\lambda_i \in K-\{0,1\}$ are pairwise different.

Let us consider the non-singular complex projective algebraic curve

\begin{equation} \label{defineGFC}
C^{k}(\lambda_{1},\ldots,\lambda_{n-2}):=\left \{ \begin{array}{rcc}
              x_0^k+x_1^k+x_2^k&=&0\\
              \lambda_1x_0^k+x_1^k+x_3^k&=&0\\
              \vdots\hspace{1cm} &\vdots &\vdots\\
              \lambda_{n-2}x_0^k+x_1^k+x_n^k&=&0\\
             \end{array}\right \}\subset {\mathbb P}^n.
\end{equation}

\s
\noindent
\begin{remark}[$C^{k}(\lambda_{1},\ldots,\lambda_{n-2})$ as a fiber product of classical Fermat curves]
Set $\lambda_{0}=1$ and, for each $j \in \{0,1,\ldots,n-2\}$, let $C_{j}$ be the classical Fermat curve defined by $\lambda_{j}x_{1}^{k}+x_{2}^{k}+x_{3+j}^{k}=0$. Let us consider the rational maps $\pi_{j}:C_{j} \to {\mathbb P}^{1}=K \cup \{\infty\}$ defined by $\pi_{j}([x_{1}:x_{2}:x_{3+j}])=-(x_{2}/x_{1})^{k}$. The branch values of $\pi_{j}$ are $\infty$, $0$ and $\lambda_{j}$. If we consider the fiber product of all these curves, with the given maps, we obtain a reducible projective algebraic curve with $k^{n-2}$ irreducible components. All of these components are isomorphic to $C^{k}(\lambda_{1},\ldots,\lambda_{n-2})$.
\end{remark}

\s

Let $H_{0}$ be the group generated by the linear transformations $\varphi_{0},\ldots,\varphi_{n}$, where
\begin{center}
 $\varphi_j([x_0:\cdots:x_j:\cdots:x_n]):=[x_0:\cdots:w_kx_j:\cdots:x_n]$,\end{center}
where $w_k$ is a primitive $k$-th root of unity. In \cite{GHL09} the following facts were proved:
\begin{enumerate}
\item $H_{0} \cong {\mathbb Z}_{k}^{n}$.
\item $\varphi_{0} \circ \varphi_{1} \circ \cdots \circ \varphi_{n}=1$. 
\item $H_{0}<{\rm Aut}(C^{k}(\lambda_{1},\ldots,\lambda_{n-2}))$.
\item The set ${\rm Fix}(\varphi_{j})$ of fixed points of  $\varphi_j$ in $C^{k}(\lambda_{1},\ldots,\lambda_{n-2})$ is given by the intersection 
$${\rm Fix}(\varphi_j):=\{x_j:=0\} \cap C^{k}(\lambda_{1},\ldots,\lambda_{n-2}),$$
which is of cardinality $k^{n-1}$. Set $F(H_{0}):=\cup_{j=0}^n{\rm Fix}(\varphi_j)$.
\item The map 
\begin{equation} \label{pi0}
\pi_{0}:C^{k}(\lambda_{1},\ldots,\lambda_{n-2}) \to {\mathbb P}^{1}: [x_0:\cdots:x_j:\cdots:x_n] \mapsto -\left(\frac{x_{1}}{x_{2}}\right)^{k}
\end{equation}
is a regular branched cover with deck group $H_{0}$ and whose branch values are $$\infty, 0,1, \lambda_{1}, \ldots, \lambda_{n-2},$$ each one of order $k$. In other words, the pair $(C^{k}(\lambda_{1},\ldots,\lambda_{n-2}),H_{0})$
is a generalized Fermat pair of type $(k,n)$. 
\end{enumerate}

\s
\noindent
\begin{theorem}\label{teo1}
The generalized Fermat pairs $(F_{k,n},H)$ and $(C^{k}(\lambda_{1},\ldots,\lambda_{n-2}),H_{0})$ are equivalent.
Moreover, the only non-trivial elements of $H_{0}$ acting with fixed points are the non-trivial powers 
of the generators $\varphi_{0}, \ldots, \varphi_{n}$.
\end{theorem}
\begin{proof}
This result was obtained, for $K={\mathbb C}$ in \cite{GHL09}.

It can be seen that a generalized Fermat curve $F_{k,n}$  is in fact a 
fiber product of $n-1$ classical Fermat curves. In fact, the $n-1$ triples 
$$\{\infty,0,1\}, \{\infty,0,\lambda_{1}\}, \ldots, \{\infty,0,\lambda_{n-2}\}$$
produce, respectively, the Fermat curves
$$C_{0}: x_{1}^{k}+x_{2}^{k}+x_{3}^{k}=0, 
C_{1}: \lambda_{1} x_{1}^{k}+x_{2}^{k}+x_{4}^{k}=0, \ldots,
C_{n-2}: x_{1}^{k}+x_{2}^{k}+x_{n}^{k}=0.
$$

If we set $\lambda_{0}=1$, then on $C_{j}$ we consider the map $\pi_{j}:C_{j} \to {\mathbb P}^{1}$ defined by
$\pi([x_{1}:x_{2}:x_{3+j}])=-(x_{2}/x_{1})^{k}$. The branch values of $\pi_{j}$ are $\infty, 0, \lambda_{j}$.
If we consider the fiber product of the above curves, using the above maps, we obtain a reducible algebraic 
curve admitting a group $({\mathbb Z}_{k}^{2})^{n-1}$ as a group of automorphisms and $k^{n-2}$ 
irreducible components. All its irreducible components are isomorphic and they are  generalized Fermat 
curves of type $(k,n)$, each one is invariant by a subgroup isomorphic to ${\mathbb Z}_{k}^{n}$. 
Let $\widehat{C}$ be one of these irreducible components and let $\widehat{H}$ be its stabilizer
in the above group.  Then the quotient $\widehat{C}/\widehat{H}=F_{k,n}/H$. Now, 
the universality property of the fiber product ensures that $(F_{k,n},H)$ and $(\widehat{C},\widehat{H})$ 
are isomorphic. By the construction of the fiber product, it can be seen that in fact $(\widehat{C},\widehat{H})$ 
and $(C^{k}(\lambda_{1},\ldots,\lambda_{n-2}),H_{0})$ are isomorphic.
\end{proof}

\s
\subsection{Automorphisms of generalized Fermat curves}
Let us consider a generalized Fermat pair $(F_{k,n},H)$. By Theorem \ref{teo1} we may assume (and this will be from now on) that
$$(F_{k,n},H)=(C^{k}(\lambda_{1},\ldots,\lambda_{n-2}),H_{0}).$$

If $n=2$, then $F_{k,2}$ is an ordinary Fermat curve of degree $k$ and its 
automorphism group was studied by P. Tzermias \cite{Tze95} for $p=0$ and by H. Leopoldt \cite{Leo96} for $p>0$. These results state that $H_{0}$ is the unique generalized Fermat group of type $(k,2)$ if $k \geq 4$ (in the case $k<4$ it is unique up to conjugation).

If $n \geq 3$, then in \cite[Cor. 9]{GHL09}
it was proved that, for $K={\mathbb C}$, every automorphism which normalizes $H_{0}$ is linear i.e., 
the normalizer of $H_{0}$ is a subgroup of $\mathrm{PGL}_{n+1}(\mathbb{C})$. The arguments  are 
still valid for any characteristic.

Again, assuming $K={\mathbb C}$, the following uniqueness results of the generatlized Fermat groups are known. In the case that $k=2$ (these are also called generalized Humbert curves) it was proved in \cite{CGHR08} that for $n=4,5$ the generalized Fermat group of type $(k,n)$ is unique.
In \cite{FGHL13} Y. Fuertes,  G. Gonz\'alez-Diez, the first and third author proved that for $k \geq 3$
and $n=3$ the generalized Fermat group of type $(k,n)$ is also unique. In the same paper 
it was conjectured that the uniqueness holds for $(k-1)(n-1)>2$ (in particular, that it is normal
in the whole automorphism group).  Here  we solve positively such a conjecture.

\s
\noindent
\begin{theorem}\label{unicidad-p}
Let $k,n \geq 2$ be integers so that $(k-1)(n-1)>2$. If $p>0$, then we also assume
that $k-1$ is not a power of $p$ and that $(p,k)=1$. Then $H_{0}$ is the only generalized Fermat group of 
type $(k,n)$ of $F_{k.n}$. Moreover, ${\rm Aut}(F_{k,n})$ is linear and it consists of matrices such that only 
an element in each row and column is non-zero. 
\end{theorem}

\s
\noindent
\begin{remark}
If $k-1$ is a power of $p$, the previous theorem is, in general, false. For example, if 
 $n=2$ and $k=1+p^h$, $p>0$,  the group $H_0$ is not always  a normal subgroup of 
\red{$\mathrm{Aut}(F_{k,n})=\mathrm{PGU}_{3}(p^{2h})$.}
\end{remark}

\s
\noindent
\begin{corollary}
Every generalized Fermat curve of type $(k,n)$ has a unique generalized Fermat group of same type if $(k-1)(n-1)>2$ and, for $p>0$, that $k-1$ is not a power of $p$.
\end{corollary}

\s
\noindent
\begin{corollary}
Let $k>2$ and, for $p>0$, let us assume that $(p,k)=1$ and that $k-1$ is not a 
power of $p$. Then $H_{0}$ is a normal subgroup of ${\rm Aut}(F_{k,n})$.
\end{corollary}

\s
\noindent
\begin{remark}
If $(k-1)(n-1) \leq 2$, then it is known that 
${\rm Aut}(F_{k,n})<{\rm PGL}_{n+1}(K)$. 
Let us now assume that $(k-1)(n-1)>2$ and, for $p>0$, that $k-1$ is not a power of $p$. 
Theorem \ref{unicidad-p} asserts that ${\rm Aut}(F_{k,n})$ coincides with the normalizer 
$N(H_{0})$ of $H_{0}$, so by the results in  \cite[Cor. 9]{GHL09} we obtain ${\rm Aut}(F_{k,n})<{\rm PGL}_{n+1}(K)$. In the same paper it is mentioned how to compute ${\rm Aut}(F_{k,n})$. This is done observing the short exact sequence:
\[
1 \rightarrow H_{0} \rightarrow {\rm Aut}(F_{k,n}) \rightarrow G_0 \rightarrow 1,
\]
where $G_0$ is the subgroup of $\mathrm{PGL}_{2}(K)={\rm Aut}(\mathbb{P}^1)$ which leaves invariant the set of branch points $\{0,1,\infty,\lambda_1,\ldots, \lambda_{n-2}\}$. 
\end{remark}

\s

In the case that $K={\mathbb C}$, the above uniqueness results provides the following ``kind of Torelli's" result. 

\s
\noindent
\begin{corollary}
Let $\Gamma_{1}, \Gamma_{2} < {\rm PSL}_{2}({\mathbb R})$ be Fuchsian groups acting on the upper-half plane ${\mathbb H}^{2}=\{z \in {\mathbb C}: {\rm Im}(z)>0\}$ so that ${\mathbb H}^{2}/\Gamma_{j}$ has signature $(0;k,\stackrel{n+1}{\ldots},k)$. Let $\Gamma_{j}'$ be the commutator subgroup of $\Gamma_{j}$. If $\Gamma'_{1}=\Gamma'_{2}$, then $\Gamma_{1}=\Gamma_{2}$.
\end{corollary}

\s

Theorem \ref{unicidad-p} states that if $(k-1)(n-1)>2$ (and $k-1$ not a power of $p$ in the case $p>0$), then ${\rm Aut}(F_{k,n})$ is a linear group. The following states a similar result for the case that $p>0$ and $k-1$ is a power of $p$ under an extra condition.

\s
\noindent
\begin{theorem}\label{teolineal}
Let $p>0$, $(p,k)=1$ and assume that $k-1$ is a power of $p$.
If $(k,n+1)=1$, then ${\rm Aut}(F_{k,n})$ is a subgroup of $\mathrm{PGL}_{n+1}(K)$ and it 
consists of elements $A=(a_{ij})$ such that 
\[
A^t \Sigma_i A^q = \sum_{\mu=0}^{n-2} b_{i,\mu}\Sigma_{\mu},
\]
for a $(n-1)\times (n-1)$ matrix $(b_{i,\mu})$, where $\Sigma_i$ are certain $(n+1)\times (n+1)$ matrices, defined in eq. (\ref{sigma-def}).
\end{theorem}

\s




\section{Hyper-osculating points of $C^{k}(\lambda_{1},\ldots,\lambda_{n-2})$}\label{Sec:3}
In this  section we demonstrate, in characteristic zero or in characteristic $p>k^{n-1}$, that the set $F(H_{0})$  of fixed points of the generalized Fermat 
group $H_{0}$ coincides with the set of hyper-osculating points of the curve $F_{k,n}=C^{k}(\lambda_{1},\ldots,\lambda_{n-2})$.

We begin by explaining the theory of hyper-osculating points of curves over
fields of characteristic $0$ following essentially  \cite{GrHa94}. In positive characteristic a variety of 
new, very interesting phenomena appear.
Also all definitions need appropriate modification in order to work. 
For the positive characteristic case
we will follow the Laksov approach \cite{Lak81}, \cite{Lak84},
since his theory was successful in giving a version of the 
generalized Pl\"ucker formulas. 

Essentially the results of Laksov, for the case of generalized Fermat 
curves, show that if we assume that the characteristic $p>k^{n-1}$, 
then we have exactly the same behavior as in characteristic $0$.

For a curve $C$ (non-singular, projective) defined over a field $K$ we 
consider the function field $K(C)$ which plays the role of the field of 
meromorphic functions.  The points of the curve can be seen as places 
(equivalence classes of valuations) and a function $f$ in $K(C)$ is called 
{\em holomorphic} at $P$ if $v_P(f) \geq 0$, where $v_P(f)$ is the valuation of 
$f$ at $P$. Holomorphic  functions admit Taylor  expansions at the completions of the valuation rings. For the general theory of functions fields over arbitrary 
fields we refer to \cite{Sti93}, \cite{Gol13}.

\s
\subsection{Preliminaries on hyper-osculating points of curves.}
Let $C$ be a projective smooth curve of the projective space ${\mathbb P}^n$. Let us consider an $s$-plane
$\Pi \subset {\mathbb P}^n$, $1\leq s\leq n-1$, 
and let us define the  multiplicity of $\Pi$ in $P \in C$ as
\begin{center}
${\rm mult}_{P}(\Pi \cap C):=\mbox{Order of contact of  $\Pi$ and $C$ in $P$}.$  
\end{center}

It is known that there exists a unique $s$-plane, denoted by $\Pi(s,P)$, such that 
$${\rm mult}_{P} (\Pi(s,P) \cap C)\geq s+1,$$ and that there exists at most a finite number of points $P\in C$ such that
$${\rm mult}_{P} (\Pi(s,P) \cap C)> s+1.$$

 The $s$-plane $\Pi(s,P)$ is called the {\it osculating $s$-plane of $C$ at $P$} and a point $P\in C$ is called a {\it hyper-osculating point} if 
\begin{center}
 ${\rm mult}_{P}(\Pi(n-1,P) \cap C)>n$. 
\end{center}

\s
\noindent
\begin{remark}
\label{re:inv-aut-pn}
  Let $\varphi \in {\rm Aut}({\mathbb P}^n)\cong {\rm PGL}_{n+1}(K)$. Observe that  
 \begin{center}
 ${\rm mult}_{\varphi(P)}(\varphi(\Pi) \cap \varphi(C))={\rm mult}_{P}(\Pi \cap C)$.
 \end{center}
In particular, $P$ is a hyper-osculating point of $C$ if and only if $\varphi(P)$ is a hyper-osculating point of 
 $\varphi(C)$.
\end{remark}

\s

\subsection{Laskov's theory of osculating planes}
Let $C$ be a  smooth curve of genus $g$ over a general field $K$ and let
$D$ be  a  divisor in $C$. Moreover, let $V$ be a linear system 
in $H^0(C,D)$ of projective dimension $n$. We note $\deg D$ the degree of the divisor $D$.\\

Tensor powers of the 
sheaf of differentials can be interpreted as 
\[
 (\Omega_C^1)^{\otimes m} = I^m/I^{m+1}, 
\]
where $I$ is the ideal defining the diagonal in the product $C\times C$. 
Let $p,q$ be the two projections $C\times C$ into the first and second 
factor respectively. Laksov defined the bundle of principal parts 
$P^m(D)=p_*(q^* \mathcal{O}_C(D) | C(m))$, where $C(m)$ is the subscheme of 
$C\times C$ 
defined by $I^{m+1}$. 
He then introduced a family of maps 
\[
v^m(D): H^0(C,D)_C:=H^0(C,D)\otimes_{K} \mathcal{O}_C \rightarrow P^m(D)
\]
and the corresponding map $v^m:V_C:=V\otimes_{K} \mathcal{O}_C \rightarrow P^m(D)$. Let $B^m$ and $A^m$ be the image and the cokernel of $v^m$.
The Corollary $2$ of \cite{Lak84} implies that  there are integers 
\[
0=G_0 < G_1 < \cdots < G_{n} \leq \deg D  < G_{n+1}=\infty
\] such that  $\mathrm{rank}B^j=(s+1)$ for  $G_s \leq j < G_{s+1}$. The above
sequence is called the {\em gap sequence} of the linear system $V$. If $G_m=m$ for 
$m=0,1,\ldots,n$ then the gap sequence is called {\em classical}. 

\s
\noindent
\begin{definition}[Associated Curves]
The surjection $V_C\rightarrow A^s$ induced by the map $v^{b_s}$ defines a 
map 

 \[
 f_s : C \rightarrow \mathbb{G}(s,n)\]
to the grassmanian of $s$-planes in $\mathbb{P}^n$. The $s$-plane in $\mathbb{P}^n$ is 
called the {\it associated $s$-plane} to $V$ at $P$, and the degree 
$d_s$ of the map $f_s$ is called the $s$-rank of the linear system $V$.

The grassmanian can be embedded in terms of the Pl\"ucker coordinates in a projective space 
$\mathbb{P}^N$, where $N=\binom{n+1}{s+1}-1$.
We will denote by $b_s(P)$ the ramification index and by $b_s$ the sum of all ramification indices of 
the composition $C \stackrel{f_s}{\longrightarrow} 
\mathbb{G}(s,n) \rightarrow \mathbb{P}^N$. The image of the later map is 
called the $s$-associated curve. 
\end{definition}

\s
\noindent
\begin{remark}
Geometrically $d_s$ can be interpreted as the number of associated $s$-planes to $V$ which intersect a generic  $(n-s-1)$-plane 
  of $\mathbb{P}^n$.  In addition, we have that  $d_s=\mathrm{rank} A^s$.  See section $5$ of the article \cite{Lak84}.
\end{remark}

\s

Let $e_0,e_1,\ldots, e_n$ be a basis of $V$.  Using the canonical maps
$v^0: V_C\rightarrow \mathcal{O}_C(D)$, we can prove that this basis induces a set of  linearly independent functions  
$v_0, v_1,...,v_n$ belonging to the local ring $\mathcal{O}_{C,P}$, $P\in C$, such that there exists a sequence 
of integers $h_0<h_1<\cdots <h_n$, where $h_i:=\mathrm{Ord}_{P}v_i$.
These integers are called the {\it Hermitian invariants at $P$}.

The $s$-plane associated to the sub-space of $V$ 
spanned by $e_{s+1},...,e_{n}$ is the unique $s$-plane with maximal contact order with $V$ at $P$ 
(the order of contact is equal to $h_{s+1}-h_{0}$). 
This $s$-plane is called the {\it osculating  $s$-plane to $V$ at $P$}.

Let $C$ be a projective smooth curve of the projective space ${\mathbb P}^n$. If
$f_0:C \rightarrow {\mathbb P}^n$ is the natural embedding defined by the
inclusion $C \subset {\mathbb P}^n$, and the divisor $D$ is the inverse image of a 
hyperplane $\Pi$ of  $\mathbb{P}^n$, we obtain that $h_0=0$ for all $P\in C$ 
and that the concepts of {\it osculating $s$-plane to $V$ at $P$} and {\it osculating  $s$-plane of $C$ at $P$} 
coincide.

Additionally, given a local uniformizer $z$ at the point $p$, the normal form of $f_0$ in $P$ is obtained in the following manner:  
$$f_0(z):=[v_0(z):\cdots: v_{n}(z)].$$


When the characteristic $p$ is small, then a lot of new phenomena appear, 
however for  $p>\deg D$ the situation is similar as in characteristic zero:

\s
\noindent
\begin{theorem}[See \mbox{\cite[Th. 15]{Lak84}}] \label{th:Laksov} 
 Assume that the characteristic $p$ of the ground field is zero or strictly grater than $\deg D$. 
 Fix a point $P\in C$ and let 
 $h_0,  h_1,\ldots, h_n$
 be the Hermite invariants of the linear system $V$ at $P$. Then:
 \begin{enumerate}
  \item The linear system $V$ has classical gap sequence, i.e. $G_m=m$ for $m=0,1,\ldots,n$. 
  \item The ramification index $b_s(P)$ of $f_s$ at $P$ is equal to $h_{s+1}-h_s-1$ for $s=0,1,\ldots,n-1$. 
  \item  The Pl\"ucker formulas take the form: 
  \[
   d_{s+1} - 2d_s + d_{s-1} =(2g-2)- b_s \mbox{\;for } s=0,1,\ldots,n-1,
  \]
  where $d_{-1}=0$ and $d_{n}=0$.
 \item The osculating and associated $s$-planes to $V$ at $P$ coincide. 
 \end{enumerate}
\end{theorem}

\s


\subsection{The hyper-osculating points of $F_{k,n}$}
 Let $f_0:F_{k,n} \rightarrow {\mathbb P}^n$ be the natural embedding defined by the
inclusion $F_{k,n}\subset {\mathbb P}^n$. Let $P$ be a point in $F(H_0)$ and let $z$ be a local uniformizer at $P$.
The following lemma helps us to find the normal form of $f_0$ around $z(P)=0$.

 Let $\Pi$ be a hyperplane section of the projective space ${\mathbb P}^n$ and $D= f_0^{\star}(\Pi)$
the inverse image divisor of $\Pi$. 
Using the Bezout theorem we obtain that $\deg D=k^{n-1}$. 
In the rest of this section Theorem \ref{th:Laksov} will be used quite a lot, 
for this reason we will impose, as a general hypothesis in the entire rest of the section, 
that the characteristic of the ground field is zero, or stricly greater than $k^{n-1}$.

\s
\noindent
\begin{lemma}
\label{le:for-nor}
Let us conserve the previously defined notations. Assume that we are working over a field of zero  characteristic  or stricly greater than $k^{n-1}$.
Then there exists a sequence of $n-1$ integers, 
\[1=l_0<2=l_1<l_2<\cdots <l_j<\cdots <l_{n-2}\leq k^{n-2},\]
 such that the normal form of $f_0$ around $z(P)=0$ is the following:
 \begin{center}
 $f_0(z)=[1:z: g_0(z^{k}):g_1(z^{k}):\cdots :g_{i}(z^{k}):\cdots : g_{n-2}(z^{k})]$
\end{center}
where the   $g_i$ admit an expansion  $g_i(z)=z^{l_i}+\cdots+\cdots$. 
\end{lemma}
\begin{proof}
 We will begin by the case of the characteristic of the field being zero.\\

 Using linear substitutions in the system of equations which define the
curve $F_{k,n}=C^{k}(\lambda_{1},\ldots,\lambda_{n-2})$,  followed by an automorphism of ${\mathbb P}^{n}$, 
we can suppose that $F_{k,n}=C^{k}(\hat{\lambda}_{1},\ldots,\hat{\lambda}_{n-2})$ and that $P\in {\rm Fix}(\varphi_1)$, 
These transformations do not affect the condition of being or not being a point of hyper-osculation,
see Remark \ref{re:inv-aut-pn}.

\s
 
In order to simplify the notations, we say that $\hat{\lambda}_0=1$.
Then the point $P$ in homogeneous coordinates is 

\begin{center}
 $P:=[1:0:\rho_1:\rho_2:\cdots:\rho_{n-1}]$,
\end{center}
where $\rho_i^{k}=-\hat{\lambda}_{i-1}$, $0\leq i\leq n-1$.\\

Let $f_0:F_{k,n}\rightarrow {\mathbb P}^n$ be the natural embedding defined by the inclusion $F_{k,n}\subset {\mathbb P}^n$,
and let us consider the following Taylor series centered in $t=0$:

\begin{center}
$\sqrt[k]{1+t}=\sum_{i=0}^{\infty} \binom {k^{-1}} {i}t^i$, $|t|<1$,
\end{center}
where 
\begin{equation} \label{binom1}
\binom{k^{-1}}{i}:=\frac{\Gamma(k^{-1}+1)}{\Gamma(i+1)\Gamma(k^{-1}-i+1)}=
\prod_{\nu=1}^i \frac{k^{-1}+1-\nu}{\nu} =  \frac{1}{i! k^i}\prod_{\nu=1}^{i-1}
(1- k\nu)
 \in \mathbb{Q}.
\end{equation} 

\s
\noindent
\begin{remark}
The binomial coefficient $\binom{n}{i}$ for $n,i \in \mathbb{N}$ has always meaning in fields of positive characteristic $p$, since we can  always reduce it modulo $p$. The binomial coefficients  in eq. (\ref{binom1}) are not defined 
if $p \leq  i$. 
\end{remark}

\s
\noindent
\begin{remark} \label{re:non-zero-bin}
If $Mk<p$ then $\binom{k^{-1}}{i}\neq 0$ for all 
$i<M$. 
Indeed, 
by eq. (\ref{binom1}) we observe that for $1\leq \nu  \leq i-1<M$ the quantity $k\nu-1 \not \equiv 0 \mod p$. 
Otherwise, $ 0< \mu p = k\nu-1 < p/M \cdot i -1<p$ for $\nu,\mu \in \mathbb{N}$, a contradiction. 
\end{remark}

\s

Using this expansion, we can describe $f_0$ explicitly in a neighborhood of $P$. 
 Let $z$ be a local uniformizer at $P$, we express locally 
\[ f_0(z)=
\left[
1:z:\sum_{i=0}^{\infty} c_{(i,1)}z^{ik}:\sum_{i=0}^{\infty} c_{(i,2)}z^{ik}: \cdots: 
  \sum_{i=0}^{\infty} c_{(i,n-1)}z^{ik}
\right]
  ,\]
where $\displaystyle c_{(i,j)}:=\frac{\rho_{j}}{\hat{\lambda}_{j-1}^i}\binom{k^{-1}}{i}$, $1\leq j\leq n-1$, $i\geq 0$.

We can prove by induction on $j$,  that for each integer  $1\leq j\leq n-2$, there exists a sequence of $n-2$ integers
 \begin{center}
 $1=l_0<2=l_1<l_2 <\cdots <l_j\leq \cdots \leq \cdots  \leq l_{n-2}$,
 \end{center}
 for which there exists a change of coordinates of ${\mathbb P}^{n}$ (which is to say, an automorphism of ${\mathbb P}^n$)
 such that
\[ f_0(z)=\left[1:z:\sum_{i=1}^{\infty} d_{(i,1)}z^{ik}:\sum_{i=2}^{\infty} d_{(i,2)}z^{ik}: 
  \sum_{i=l_2}^{\infty} d_{(i,3)}z^{ik}:\cdots: \sum_{i=l_{n-2}}^{\infty} d_{(i,n-1)}z^{ik} \right],
  \]
   where $d_{(l_{m-1},m)}=1$ for all  $1\leq m \leq n-2$.\\
   
   By virtue of part $(iii)$ of theorem $10$ of  \cite{Lak84}
   we obtain that the Hermite invariant $h_{n}$ is less than or equal to $\deg D= k^{n-1}$ (It is worth mentioning that this 
   result is valid in the case of the positive characteristic). Implying that $l_{n-2}\leq k^{n-2}$.
   This will prove  the lemma in  the case of characteristic zero.\\
   
   Using the fact that $h_{n}\leq k^{n-1}$, and Remark \ref{re:non-zero-bin}, we can ensure that for fields of characteristic
   $p$ such that $k^{n-1}<p$ the method of recurrence raised previously functions in the same way. 
   However the sequence of integers  $l_2<l_3<\cdots < l_{n-2}$  
   obtained in the case of the positive characteristic could differ from the sequence of integers obtained in  
   the case of characteristic zero.\\
   
Let us now do some steps of the  induction in order to indicate some problems
that may occur over fields of positive characteristic:

{ \small \[ f_0(z)=\left[1:z:c(0,1)+c(1,1)z^k+\cdots : c(0,2)+c(1,2)z^k+\frac{\rho_1}{\hat{\lambda}^2_{1} }
\binom{k^{-1}}{2} z^{2k} +\cdots: \ldots : 
  \right],
\]
}

In the first step we subtract the constant function $1$ multiplied by 
$c(0,i)$ from all but the first two  projective coordinates of $f_0(z)$ arriving at 
\[
f_0(z)=\left[1:z:c(1,1)z^k+\cdots : c(1,2)z^k+\frac{\rho_2}{\hat{\lambda}^2_{1} }
\binom{k^{-1}}{2} z^{2k} +\cdots: \ldots : 
  \right],
\] 

The coefficient $c(1,1)=\frac{\rho_1}{\hat{\lambda}_0} \binom{k^{-1}}{1} \neq 0$
so we can divide the third coordinate of $f_0(z)$ by $c(1,1)$ in order 
to have coefficient of $z^k$ equal to $1$. Then we subtract from 
all but the first two coefficients the third coefficient in order to 
eliminate the term $z^k$.  The coefficient 
of $z^{2k}$ in the fourth coordinate equals to 
\begin{eqnarray*}
c(1,1)c(2,2)-c(1,2)c(2,1) & = & \frac{\rho_1}{\hat{\lambda}_0} \binom{k^{-1}}{1} 
\frac{\rho_2}{\hat{\lambda}^2_{1} }
\binom{k^{-1}}{2} -\frac{\rho_2}{\hat{\lambda}_1} \binom{k^{-1}}{1}
\frac{\rho_1}{\hat{\lambda}_0^2} \binom{k^{-1}}{2} \\
 & = & 
 \frac{\rho_1 \rho_2}{ \hat{\lambda}_{0}\hat{\lambda}_{1} } \binom{k^{-1}}{2} \binom{k^{-1}}{1}
 \left(  
\frac{1}{\hat{\lambda}_1} -\frac{1}{\hat{\lambda}_0} 
 \right) \neq 0, \\
\end{eqnarray*}
since 
\[
\binom{k^{-1}}{2}=
\frac{k^{-1} (k^{-1}-1)}{2}=\frac{1-k}{2 k^2}\neq 0 \;\mbox{ and}\; \hat{\lambda}_1\neq 1
\]

We can now normalize the coefficient of $z^{2k}$ to $1$ and subtract it multiplied by the appropriate constant  from the 
next coordinate. Doing this subtraction it can happen that  the coefficients of 
$z^{3k},z^{4k}$ etc are also eliminated. So we set $2<l_2$ the first non zero exponent in the above subtraction. We then  
proceed in a similar way until all coordinates are in the form requested by the lemma. 
\end{proof}

The next theorem describes the hyper-osculating points of $F_{k,n}$ and the ramification indices.

\s
\noindent
\begin{theorem}
\label{th:hyperosc} 

Assume that the characteristic $p$ of the ground field is zero or strictly grater than $k^{n-1}$.
 Let  $(n-1)(k-1)>2$.  Then the following  holds: 
 \s
 \begin{enumerate}
 \item The set of hyperosculating points of $F_{k,n}$ is the set $F(H_0)$.
 \s

  \item  If $P \in F(H_0)$, then $b_1(P)=k-2$ and $b_l(P)=k-1$  for all $2\leq l\leq n-1$.
   \end{enumerate}

\end{theorem}

\s

The following corollary is directly derived from Theorem \ref{th:hyperosc} and Lemma \ref{le:for-nor}.

\s
\noindent
\begin{corollary}
\label{co:hyperosc}
 Let $z$ be  a local chart of  $F_{k,n}$  around a point  $P$. Then the normal form of 
 $f_0$ in  $z(P):=0$  is:
 \begin{enumerate}
  \item If $P \in F(H_0)$
 \[
f_0(z)=
\left[1:z: g_0(z^{k}):g_1(z^{k}):\cdots :g_{i}(z^{k}):\cdots : g_{n-1}(z^{k})
\right],
\]
where the  $g_i$ are holomorphic functions such that $g_i(z)=z^{i+1}+\cdots+\cdots$,
\item If $P \not \in F(H_0)$, then 
\[ 
f_0(z)=\left[1:z: z^2+\cdots :\cdots : z^{(n-1)}+\cdots \right].
\]
 \end{enumerate}
\end{corollary}

\s

\begin{proof}[Proof of the Theorem \ref{th:hyperosc}.]

Let $P$ be a point in $F(H_0)$. 
Using  part $2$ of Theorem \ref{th:Laksov} and  Lemma  \ref{le:for-nor}, we obtain the following system
of equations:
 
\begin{center}
 $\left \{ \begin{array}{lcc}
           2+b_1(P)& =& k\\
           3+b_1(P)+b_2(P) & = & 2k \\
           4+b_1(P)+b_2(P)+b_3(P) & = & l_2k\\
           \vdots \hspace{1cm} \vdots \hspace{1cm} \vdots \hspace{1cm} \vdots\hspace{0.9cm} \ddots  & \vdots  & \vdots\\
           n+b_1(P)+b_2(P) +b_3(P)+ \cdots +b_{n-1}(P) &= &l_{n-2}k
          \end{array} \right. $
\end{center}
Equivalently, we obtain
\begin{center}
 $\left \{ \begin{array}{ccl}
            b_1(P)& =& k-2\\
            b_2(P)& = & k-1 \\
            b_3(P)& = &(l_2-2)k-1\\
            \vdots & \vdots & \hspace{1cm}\vdots\\
            b_{n-1}(P) & = &(l_{n-2}-l_{n-3})k-1
           \end{array} \right. $

\end{center}

 Observe that  $b_l(P)\geq k-1$ for all $2\leq l\leq n-1$. 
 In particular, $P$ is  a hyper-osculating point.\\

 Since the cardinality of $F(H_0)$ is equal to  $(n+1)k^{n-1}$, we have the following lower bound from the total
ramification indices:
\begin{center}

 $\left \{ \begin{array}{cclc}
            b_1& = & \hat{b}_1:=(n+1)k^{n-1}(k-2)\\
            b_l& \geq & \hat{b}_l:=(n+1)k^{n-1}(k-1) & \mbox{for every }\;2\leq l\leq n-1 
            
           \end{array} \right. $
\end{center}

Observe that in order to finish the demonstration of the theorem, it is necessary and sufficient to prove 
$b_l=\hat{b}_l$, for all $1\leq l\leq n-1$. We will now  prove these equalities.\\

 Consider the following inequality
 
 \begin{center}
  $\displaystyle 0\leq b_l-\hat{b}_l\leq \sum_{l=0}^{n-1}(n-l)(b_l-\hat{b}_l)$,
 \end{center}
 where $b_0=\hat{b}_0=0$. The idea is to show that the right part of the inequality is zero.\\

Remember that the genus of $F_{k,n}$ is given by the following formula:
\[
g_{(k,n)}:=\frac{k^{n-1}(((n-1)(k-1)-2)+2}{2}.
\] 
  Via direct calculation, we obtain the following equality:

\[\sum_{l=0}^{n-1}(n-l)\hat{b}_l=n(n+1)(g_{(k,n)}-1)+(n+1)k^{n-1}.
\]
Using the Pl\"ucker formulas  (part 2 of Theorem \ref{th:Laksov}), we obtain 
\begin{eqnarray*} 
\sum_{l=0}^{n-1}(n-l)b_l & = &\sum_{l=0}^{n-1}(n-l)(2(g_{(k,n)}-1)-\Delta^2d_{l}) \\ &= &n(n+1)(g_{(k,n)}-1)
-\sum_{l=0}^{n-1}((n-l)\Delta^2d_l),
\end{eqnarray*}
where $ \Delta^2 d_{l}=d_{l+1}-2d_l+d_{l-1}$.

From a simple calculation it is obtained that
\[
\sum_{l=0}^{n-1}(n-l)\Delta^2d_l=d_{n}-(n+1)d_1+nd_{-1}.
\]
Since  $d_n=d_{-1}=0$ and $d_1=k^{n-1}$, therefore
\[
\sum_{l=0}^{n-1}(n-l)b_l=n(n+1)(g_{(k,n)}-1)+(n+1)k^{n-1},
\]
which implies that $b_l=\hat{b}_l$ for all $1\leq l\leq n-1$. 
\end{proof}

\s

\section{Complete intersections and linear automorphisms}\label{Sec:5}
Let $\mathbb{P}^{n}$ be the projective space with homogeneous coordinates 
$[x_1:\cdots:x_{n+1}]$. 
Consider the curve $F_{k,n}=C^{k}(\lambda_{1},\ldots,\lambda_{n-2})$ embedded in $\mathbb{P}^{n+1}$   as the  intersection of the $n-1$ hypersurfaces 
$f_{i}:=\lambda_i x_{1}^k+z_{2}^k+z_{3+i}^k=0$ for  $0\leq i \leq n-2$, where $k,n \geq 2$ are integers so that, for $p>0$,  $(k,p)=1$ (see eq. \eqref{defineGFC}).

\s
\noindent
\begin{proposition} \label{sing-set}
The curve $F_{k,n}$ is a nonsingular complete intersection.
\end{proposition}
\begin{proof}
The curve is given as the intersection of $n-1$ hypersurfaces 
$f_i:=\lambda_i x_1^k+x_2^k+x_{3+i}^k$ for $i=0,\ldots,n-2$. 
We consider the matrix of  $\nabla f_i$ written as rows.
\begin{equation} \label{Jac}
\begin{pmatrix}
k x_1^{k-1} & k x_2^{k-1} &  k x_3^{k-1} & 0& \ldots & 0 \\
\lambda_1 k x_1^{k-1} & k x_2^{k-1} & 0 & k x_4^{k-1}& \ldots &  0 \\
 \vdots      & \vdots    & \vdots  & \vdots  & \vdots \\
 \lambda_{n-2 } k x_1^{k-1} & k x_2^{k-1} & 0 & \ldots &  0 &
  kx_{n+1}^{k-1}
\end{pmatrix}.
\end{equation}
By the  defining equations of the curve we see that a point 
which has two variables $x_i=x_j=0$ for  $i\neq j$ and $1\leq i,j \leq n+1$ has also  $x_t=0$ for $t=1,\ldots,n+1$. Therefore the above matrix has the maximal rank $n-1$ at all points of the curve. 

So 
the defining hypersurfaces are intersecting transversally and 
the corresponding algebraic curve they define is non singular. 
\end{proof}

\s
\noindent
\begin{proposition}
The ideal $I_{k,n}$  defined by the $n-1$ equations defining $F_{k,n} \subset\mathbb{P}^{n+1}$ is prime.
\end{proposition}
\begin{proof}
We will follow the method of \cite[sec. 3.2.1]{Kon02}. Observe 
first that the defining equations $f_{0},\ldots,f_{n-2}$ form a regular sequence, 
and $K[x_1,\ldots,x_{n+1 }]$ is a Cohen-Macauley ring and the ideal $I_{k,n}$ they define    
is of codimension $n-1$. The ideal $I_{k,n}$ is prime by the Jacobian Criterion \cite[Th. 18.15]{Eis95},
\cite[Th. 3.1]{Kon02} and Proposition \ref{sing-set}. In remark \cite[3.4]{Kon02} we pointed out that an ideal $I$ is prime if the  the singular locus of the algebraic set defined by $I$  has  big enough codimension. 
\end{proof}

\s
\noindent
\begin{remark}[Stable Family]
Consider now the polynomial ring  $R_1:=K[\lambda_1,\ldots,\lambda_{n-2}]$ 
and consider the ideal $J$ generated by 
$ \prod_{i=1}^{n-2} \lambda_i (\lambda_i-1) \cdot \prod_{i< j}(\lambda_i-\lambda_j)$. We consider the localization $R$ of the polynomial ring $R_1$ with respect to the multiplicative set $R_1-J$. The affine scheme ${\rm Spec} R$ is the 
space of different points $P_1,\ldots, P_{n+1}$, and the family 
$\mathcal{X} \rightarrow {\rm Spec} R$ is a stable family of curves since it has 
non-singular fibers of genus $\geq 2$. 

By the results of Deligne-Mumford \cite[lemma I.12]{DeMu69} any automorphism of the generic fibre 
is also an automorphism of the special fibre.
Special fibres have more automorphisms,  when the ramified points \[\{0,1,\infty,\lambda_1,\ldots,\lambda_{n-2} \}\] are in such a configuration, so that a finite automorphism group of 
$\mathrm{PGL}(2,K)$ permutes them. 
\end{remark}

\s

Since $F_{k,n}$ is a projective variety, for every automorphism $\sigma \in {\rm Aut}(F_{k,n})$ there is a Zariski open covering of $F_{k,n}$, $(U_i)_{i \in I}$ such that the restriction of 
$\sigma\mid_U$ is given by $n+1$ homogeneous polynomials $g_i^{(\sigma)}$ of the same degree, i.e.
if $\bar{x}=[x_1:\cdots:x_{n+1}]$, then 
\begin{equation} \label{def-aut}
\sigma\mid_U(\bar{x})=[g^{(\sigma)}_1(\bar{x}):\cdots:
g^{(\sigma)}_{n+1}(\bar{x})],
\end{equation}
see \cite[prop. 6.20]{Mil14}.

\s

All automorphisms that come as automorphisms of the ambient projective space, i.e. 
they are represented on the whole curve $F_{k,j}$ as in eq. (\ref{def-aut}) with 
$\deg g_i=1$ for all $1\leq i \leq n+1$ are called linear and they form 
a subgroup $L$ of ${\rm Aut}(F_{k,n})$.

\s
\noindent
\begin{lemma} \label{normL}
The group $L$ is a normal subgroup of  ${\rm Aut}(F_{k,n})$. 
\end{lemma}
\begin{proof}
Consider a non linear automorphism $\tau \in  {\rm Aut}(F_{k,n})$ and a linear automorphism 
$\sigma \in L$. Since $\tau$ is not linear there is an open $U \subset F_{k,n}$ 
where $\tau$ is expressed in terms of polynomials of degree $d>1$. 

Consider the element $\sigma'=\tau \sigma \tau^{-1}$. We will show that $\sigma'$ is 
linear.   Since  the curve $F_{k,n}$ is connected, the open sets  $U$ and
 $\sigma(U)$ have non trivial intersection $V$. On this set $V$ we express 
the automorphisms  $\sigma,\tau,\sigma'$ in terms of homogenenous  polynomials
$g_{i}^{(\sigma)},g_{i}^{(\tau)},g_i^{(\sigma')}$, $1\leq i \leq n+1$,  respectively of degrees $1,d,d'$ as in eq.
(\ref{def-aut}). We have $\sigma' \tau =\tau \sigma$ and this implies for $\bar{x}\in V$ the relation 
\[[g_1^{(\sigma')} \circ g_1^{(\tau)}(\bar{x}):\cdots: g_{n+1}^{(\sigma')} \circ g_{n+1}^{(\tau)}(\bar{x})]=
[g_1^{(\tau)} \circ g_1^{(\sigma)}(\bar{x}):\cdots: g_{n+1}^{(\tau)} \circ g_{n+1}^{(\sigma)}(\bar{x})].
\]

Let $I_{k,n}$ be the ideal defining the curve $F_{k,n}$. 
For each $\bar{x}\in K^{n}$ there is a $\lambda_{\bar{x}}\in K$ such that 
\[
g_i^{(\sigma')} \circ g_i^{(\tau))}(\bar{x})=\lambda_{\bar{x}} g_i^{(\tau)} \circ g_i^{(\sigma))} 
(\bar{x})\mod I_{k,n} \mbox{ for all } 1 \leq i \leq n+1.
\] 

The left hand side has degree $d'd$ while the right hand side has degree $d$.
So if we substitute  $\mu \bar{x}$ in the above equation where $\mu^{d'}=\lambda_{\bar{x}}$
we obtain $g_i^{(\sigma')} \circ g_i^{(\tau)}=g_i^{(\tau)} \circ g_i^{(\sigma)}$ for all $1\leq i \leq n+1$ modulo the homogenous ideal $I_{k,n}$ of the curve and this in turn is possible only if $d'=\deg g_i^{(\sigma')}=1$, i.e.  $\sigma'$ is given 
 in terms of linear polynomials. 

We have proved so far that there is an open cover $(U_i)_{i\in I}$ of $F_{k,n}$ where $\sigma'$ is 
given in terms of linear polynomials. Since every element in the defining ideal of the curve 
$F_{k,n}$ has degree $>1$ this means that on the nonempty intersections $U_i \cap U_j$ 
the linear polynomials expressing $\sigma'$ should not only be equal modulo the 
defining ideal, but equal as polynomials. This proves that 
$\sigma'$ is given by linear polynomials on the whole space $F_{k,n}$ so $\sigma'\in L$.   
\end{proof}

\s

\subsection{The elements of $L$}
In this section we describe the elements on the group $L$ of linear automorphisms of the curve $F_{k,n}$.  

All automorphisms $\sigma \in L$ are linear ones, so they are given in terms of an $(n+1) \times (n+1)$ matrix:
\begin{equation}\label{Adef}
\sigma(x_i)=\sum_{\nu=1}^{n+1} a_{i,\nu} x_i.
\end{equation}

An automorphism of $V(f_1,\ldots,f_{n-2})$ is a map $\sigma$ such that if $P$ is a point in 
$V(f_1,\ldots,f_{n-2})$, then $\sigma(P)$ is in $V(f_1,\ldots,f_{n-2})$.
The following holds true:
\[
f_i\circ \sigma=\sigma^*(f_i) \in \langle f_1,\ldots,f_{n-1} \rangle. 
\]
i.e.
\begin{equation} \label{aut-cond}
f_i\circ \sigma= \sum_{\nu=1}^{n-1} g_{\nu,i} f_\nu,
\end{equation}
for some appropriate polynomials $g_i\in K[x_1,\ldots,x_{n+1 }]$.
When $\sigma \in L$, so it is linear, the polynomials  $g_{\nu,i}$ are just constants.

\s
\noindent 
\begin{theorem}
Set $Y_i=\nabla{f_i}$.
If $\sigma \in L$, then $\sigma(Y_i)$ should be a linear combination 
of elements $Y_i$.
\end{theorem}
\begin{proof}
By applying $\nabla$  to \cref{aut-cond} we have for every point on the curve
\[
\nabla (f_i\circ \sigma)(P)  =\sum_{\nu=1}^{n-1} 
\big( g_{i,\nu}(P)  \nabla f_\nu(P)  + \nabla g_{i,\nu}(P)  f_\nu(P)
\big).
\]
But $f_\nu(P)=0$ so we arrive at 
\[
\nabla (f_i \circ \sigma)(P)=\sum_{\nu=1}^{n-1} 
 g_{i,\nu}(P)  \nabla f_\nu(P) 
\]
which gives rise to 
\[
\nabla (f_i \circ \sigma)=\sum_{\nu=1}^{n-1} 
 g_{i,\nu}  \nabla f_\nu +F,
\]
where $F$ is an element in the ideal $I$. The ideal $I$ is generated by polynomials 
of degree $k$, while $\nabla f_i$ are polynomials of degree $k-1$. 
Therefore, 
\begin{equation} \label{linact}
\nabla (f_i \circ \sigma)=\sum_{\nu=1}^{n-1} 
 g_{i,\nu}  \nabla f_\nu,
\end{equation}
as polynomials in $K[x_1,\ldots,x_{n+1}]$.
\end{proof}

\s

Now the  chain rule implies that, for $\sigma \in L$, 
\begin{equation} \label{NN1}
\nabla(f_i \circ \sigma)(P) = \nabla(f_i)(\sigma(P)) \circ \sigma,
\end{equation}
where $\sigma$ is given by the $(n+1 )\times (n+1 )$ matrix $A=(a_{ij})$ 
given in eq. (\ref{Adef}).
We now rewrite eq. (\ref{NN1})  and combine it with eq. (\ref{linact})
\begin{equation} \label{linY}
\sigma^* (\nabla f_i)\circ \sigma=\nabla(f_i)(\sigma(P)) \circ \sigma=\nabla(f_i \circ \sigma)(P)=\sum_{\nu=1}^{n-1} 
 g_{i,\nu}  \nabla f_\nu.
\end{equation}
Recall that  $f_j=\lambda_j x_1^k+x_2^k+x_{3+j}^k$ for $1\leq j\leq n-2$ and 
\[Y_j=(k \lambda_j x_1^{k-1},kx_2^{k-1},0,\ldots,0,k x_{j+3}^{k-1},0 ,\ldots, 0),\]
where the third  non zero element is at the $j+3$ position.
For $1\leq i \leq n+1$ let us write 
\[
\sigma^*(x_i)=\sum_{\nu=1}^{n+1} a_{i,\nu} x_\nu.
\]
So {\tiny
\[
\sigma^*(Y_j)=k
\left( 
\lambda_j 
\left( 
\sum_{\nu=1}^{n+1 } a_{1,\nu} x_\nu 
\right)^{k-1},
 \left( 
\sum_{\nu=1}^{n+1 } a_{2,\nu} x_\nu
\right)^{k-1},0\ldots 0,
\left( 
\sum_{\nu=1}^{n+1 } a_{j+3,\nu} x_\nu
\right)^{k-1} ,0 \ldots 0
\right),
\]
}

Observe that eq. (\ref{linY}) implies that 
$\sigma^*(Y_i)$ is a linear combination of $Y_i$, which involves only 
combinations of the monomials $x_i^{k-1}$,  while the $t$-th ($t=1,2,j+3$)  coefficient of $\sigma^*(Y_i)$ involves all combinations
of the terms
\[\binom{k-1}{\nu_1,\ldots,\nu_{n+1 }} \left( a_{t,1}^{\nu_1} \cdots 
a_{t,n+1 }^{\nu_{n+1 }} \right) \cdot 
\left(x_1^{\nu_1} \cdots x_{n+1 }^{\nu_{n+1 }}\right)   \mbox{ for } \nu_1+\cdots+\nu_{n+1 }=k-1.\] 
For $\bar{\nu}=(\nu_1,\ldots,\nu_{n+1 })$  define $\mathbf{x}^{\bar{\nu}}=
x_1^{\nu_1} \cdots x_{n+1 }^{\nu_{n+1 }}$ and set 
\[
A_{t,\bar{\nu}}= a_{t,1}^{\nu_1} \cdots 
a_{t,n+1 }^{\nu_{n+1 }}.
\]
Observe that if $\binom{k-1}{\nu_1,\ldots,\nu_{n+1 }}\neq 0$ and $\mathbf{x}^{\bar{\nu}}$ does not appear as a term in the linear combination of $Y_i$, then using eq. (\ref{linY}) we have
\[
(A_{1,\bar{\nu} },\ldots, A_{n+1 ,\bar{\nu}})\cdot A=0.
\]
But $A$ is an invertible matrix so the above equation implies that 
\[
A_{t,\bar{\nu}}=0
\]
if $\mathbf{x}^{\bar{\nu}}$ does not appear as a term in the linear combination of 
$Y_i$.

\s
\noindent
\begin{lemma}
The binomial coefficients $\binom{k-1}{\nu}=0$ for all $1\leq \nu  \leq k-1$ if and only if $k-1$ is a power of the characteristic. 
\end{lemma}
\begin{proof}
The binomial coefficient $\binom{k-1}{\nu}$ is not divisible by the characteristic 
$p$ if and only if $\nu_i \leq k_i$ for all $i$, where $\nu=\sum \nu_i p^i$, 
$k-1=\sum k_i p^i$ are the $p$-adic expansions of $\nu$ and $k-1$, \cite[p. 352]{Eis95}. The result follows. 
\end{proof}

\s
\noindent
\begin{lemma}\label{lemalineal1}
Let $\sigma \in L$ given by a $(n+1)\times (n+1)$ matrix $(a_{ij})$.
If $k-1$ is not a power of the characteristic, then there is only one non-zero element 
in each column and row of $(a_{ij})$. 
\end{lemma}
\begin{proof}
If $k-1$ is not a power of the characteristic, then we see that the 
matrix $(a_{i,j})$ can have only one non zero term in each row and column. 
Indeed, if this was not true, then  for some $j$ we have two non-zero 
terms $a_{j,l_1},a_{j,l_2}$.
If $j\geq 3$, then we work with $\sigma^*(Y_{j-3})$ and for $\nu$ such that 
$\binom{k-1}{\nu}\neq 0$ we have that $a_{j,l_1}^{\nu}a_{j,l_2}^{k-1-\nu}=0$, so the desired result follows. 
\end{proof}

\s
\noindent
\begin{corollary} \label{HnormalL}
If  $k-1$ is not a power of the characteristic, then  every automorphism $\sigma\in L$ restricts to an automorphism of  the function field $K(X)$, $X=-\frac{x_2^k}{x_1^k}$, i.e. $L$ normalizes $H_{0}$.
\end{corollary}
\begin{proof}
The function field of the generalized Fermat curves can be seen as Kummer extension with 
Galois group $H$
of the rational function field $K(X)$, where $X=-\frac{x_2^k}{x_1^k}$ (see 
 \cite[par. 2.2]{GHL09} or eq. \eqref{pi0}). In order to prove that $H$ is a normal subgroup 
 of the whole automorphism group we have to show that every automorphism of the curve 
keeps the field $K(X)$ invariant.

Since there is only one non-zero element in each row and column of $A$
the automorphism $\sigma$ 
\begin{equation} \label{k+1-pow} 
\sigma^*(x_i^k)=\sum_{\nu=1}^{n+1} a_{i,\nu}^{k} x_\nu^k.
\end{equation}
Therefore
\[
\sigma^*(X)=-\frac{\sigma^*(x_2)^k}{ \sigma^*(x_1)^k}=
-\frac{ \sum_{\nu=1}^{n+1} a_{2,\nu}^k x_\nu^k} {\sum_{\nu=1}^{n+1} a_{1,\nu}^k x_\nu^k}.
\]
In the above equation we replace all variables $x_\nu$ for $\nu\geq 3$ using
the defining equations $x_\nu^k= - \lambda_{\nu-3}x_1^k-x_2^k$ in order to arrive at 
an expresion involving only $X=-\frac{x_2^k}{x_1^k}$:
\begin{eqnarray*}
\sigma^*(X) & = & 
-\frac{ a_{21}^k x_1^k +a_{22}^k x_2^k +\sum_{\nu=3}^{n+1} a_{2,\nu}^k 
\left( - \lambda_{\nu-3}x_1^k-x_2^k\right)}
 {
 a_{11}^k x_1^k +a_{12}^k x_2^k +\sum_{\nu=3}^{n+1} a_{1,\nu}^k 
\left( - \lambda_{\nu-3}x_1^k-x_2^k\right)
 } \\
 & =& 
 -\frac{  \left(-a_{22}^k +\sum_{\nu=3}^{n+1} a_{2,\nu}^k \right) X
 +\left( a_{21}^k -\sum_{\nu=3}^{n+1}  \lambda_{\nu-3} a_{2,\nu}^k \right)
 }
 { 
  \left(-a_{12}^k +\sum_{\nu=3}^{n+1} a_{1,\nu}^k \right) X
 +\left( a_{11}^k -\sum_{\nu=3}^{n+1}  \lambda_{\nu-3} a_{1,\nu}^k \right)
 }.
\end{eqnarray*}
\end{proof}

\s
\noindent
\begin{proposition}\label{powerp}
Assume that $k-1=p^h=q$ is a power of the characteristic. 
Denote by 
\begin{equation} \label{sigma-def}
\Sigma_i = \mathrm{diag}(\lambda_i,1,0,\ldots,1, 0,\ldots,0),
\end{equation}
with $1$ in the $i+3$ position. Then a matrix $A \in \mathrm{PGL_{n+1}(K)}$ corresponding to 
$\sigma \in L$ 
should satisfy
\begin{equation} \label{def-q}
A^t \Sigma_i A^q = \sum_{\mu=0}^{n-2} b_{i,\mu}\Sigma_{\mu},
\end{equation}
for a $(n-1)\times (n-1)$ matrix $(b_{i,\mu})$.
\end{proposition}
\begin{proof}
Assume that $k-1=p^h=q$ is a power of the characteristic. 
Then, 
\begin{eqnarray*}
\sigma^*(f_i) &= & \lambda_i 
\left( 
\sum_{\nu=1}^{n+1} a_{1,\nu} x_\nu 
\right)^{q+1} +
\left( 
\sum_{\nu=1}^{n+1} a_{2,\nu} x_\nu 
\right)^{q+1}+ \left( 
\sum_{\nu=1}^{n+1} a_{i+3,\nu} x_\nu 
\right)^{q+1} \\
& =& \sum_{\nu,\mu=1}^{n+1} \left(
\lambda_i a_{1,\nu} a_{1,\mu}^q + a_{2,\nu }a_{2,\mu}^q+ a_{i+3,\nu} a_{i+3,\mu}^q
 \right)x_\nu x_\mu^q \\
 & =& 
 \sum_{\nu,\mu=1}^{n+1}  B^i_{\nu,\mu}(\sigma) x_\nu x_\mu^{q}.
\end{eqnarray*}
Observe that by \cref{linact} we have 
$B^{i}_{\nu,\mu}=0$ for all $0\leq i \leq n-2, 1\leq \nu,\mu \leq n+1$, $n\neq \mu$.

The polynomials are in some sense ``quadratic forms"
\[
f_i(x_1,\ldots,x_{n+1})=(x_1,\ldots,x_{n+1}) \Sigma_i 
\begin{pmatrix}
x_1^q \\
x_2^q \\
\vdots \\
x_{n+1}^q
\end{pmatrix}
\]
so $\sigma^* f_i$ is computed as 
\[
\sigma^*f_i = (x_1,\ldots,x_{n+1}) A^t  \Sigma_i A^q
\begin{pmatrix}
x_1^q \\
x_2^q \\
\vdots \\
x_{n+1}^q
\end{pmatrix}
\]
and the above expression should be a linear combination of $f_i$. The desired result 
follows. 
\end{proof}

\s
\noindent
\begin{remark}
Matrices $A=(a_{ij})$ which satisfy eq. (\ref{def-q}) should satisfy the following equations:
For $0,\ldots,n-2$ and $1\leq \nu,\mu \leq n+1$ we 
set 
\[B_{\nu,\mu}^i= \lambda_i a_{1,\nu}a_{1,\mu}^q + a_{2,\nu} a_{2,\mu}^q +a_{i+3,\nu} a_{i+3,\mu}^q.\]
We have
\[
B_{\nu,\mu}^i =0 \mbox{ for  } \nu\neq \mu. 
\]
Moreover the coefficients $b_{i,\mu}$ in eq. (\ref{def-q})
satisfy the system 
\[
\begin{pmatrix}
1 & \lambda_1  & \lambda_2 & \cdots & \lambda_{n-2} \\
1 & 1          &   1       & \cdots & 1 \\
1 & 0      &  0 & \cdots  & 0 \\
0 & 1 & 0 &  \cdots & 0 \\
\vdots & \ddots & 1 & \ddots      & \vdots \\
\vdots &  & \ddots  & \ddots       &  0\\
0  & \cdots  & \cdots &  0  & 1
\end{pmatrix}
\begin{pmatrix}
b_{i,1} \\
b_{i,2} \\
\vdots \\
\vdots \\
\vdots \\
b_{i,n-1}
\end{pmatrix}=
\begin{pmatrix}
B_{1,1}^i \\
B_{2,2}^i \\
\vdots \\
\vdots \\
\vdots \\
B_{n+1,n+1}^i
\end{pmatrix}
\]
Which gives us that 
\[
b_{i,\nu}=B_{2+\nu,2+nu}^i=\lambda_i a_{1,2+\nu}^{q+1} +  a_{2,2+\nu}^{q+1} +a_{i+3,2+\nu}^{q+1}
\mbox{ for } 1 \leq \nu \leq n-1
\]
plus the compatibility relations
\[
\sum_{\nu=3}^{n+1} B_{\nu,\nu}^i =B_{2,2}^i
\] 
and 
\[
\sum_{\nu=3}^{n+1} \lambda_{\nu-3}B_{\nu,\nu}^i =B_{1,1}^i.
\] 
Solving these linear systems with $\lambda_1,\ldots,\lambda_{n-2}$ as parameters, seems a complicated problem, which is out of reach for now. 
\end{remark}

\s

\section{Proof of Theorem \ref{unicidad-p}}\label{Sec:7}
In this section, we assume $k,n \geq 2$ are integers so that $(n-1)(k-1)>2$ and, for $p>0$, we also assume that 
$(p,k)=1$ and that $k-1$ not a power of $p$. 

Set $F_{k,n}=C^{k}(\lambda_{1},\ldots,\lambda_{n-2})$, where $\lambda_{1},\ldots,\lambda_{n-2} \in K-\{0,1\}$ are different.

As before, let $N(H_{0})<{\rm Aut}(F_{k,n})$ be the normalizer of $H_{0}$ in the group  ${\rm Aut}(F_{k,n})$.

Lemma \ref{normL} asserts that $L$, the group of linear automorphisms of $F_{k,n}$, is a normal subgroup of ${\rm Aut}(F_{k,n})$. Corollary \ref{HnormalL} asserts that $L<N(H_{0})$ and, since $H_{0}<L$, that $H_{0}$ is a normal subgroup of $L$.

\s
\noindent
\begin{remark}
We may arrive to the same conclusion above using the 
theory of hyper-osculating points under the  condition 
 $k^{n-1}<p$ or $\mathrm{char}(K)=0$. Indeed, 
as a consequence of  Remark \ref{re:inv-aut-pn} and Theorem \ref{th:hyperosc}, we have that $L$ preserves the set of fixed points $F(H_{0})$. This in particular asserts that if $\tau \in L$, then $\tau \varphi_{j} \tau^{-1}=\varphi_{\sigma(j)}$ for a suitable permutation $\sigma$ of the set $\{0,1,\ldots,n\}$; in particular, 
$\tau H_{0} \tau^{-1}=H_{0}$. This asserts that $L<N(H_{0})$. 
\end{remark}

\s
\noindent
\begin{lemma}\label{lema2}
Under the above assumptions, $N(H_{0})=L$.
\end{lemma}
\begin{proof}
As noted above (under the assumption that $k-1$ is not a power of $p$ if $p>0$), 
Corollary \ref{HnormalL} asserts that $L<N(H_{0})$. In \cite{GHL09} it was seen that $N(H_{0})<{\rm PGL}_{n+1}(K)$ (in that article it was assumed that $K={\mathbb C}$, but the general case is seen in the same way); obtaining that $N(H_0)<L$.
\end{proof}

\s
\noindent
\begin{lemma}\label{lema3}
Under the above assumptions, $H_{0}$ is the unique generalized Fermat group of $F_{k,n}$ inside $L$. 
\end{lemma}
\begin{proof}
Let $H<L$ be another generalized Fermat group of type $(k,n)$. The group $H$ is generated by the elements $\varphi_{j}^{*}$, for $j=0,\ldots,n$, so that the non-trivial elements of $H$ acting with fixed points in $F_{k,n}$ are exactly the non-trivial powers of these generators and $\varphi_{0}^{*} \circ \varphi_{1}^{*} \circ \cdots \circ \varphi_{n}^{*}=1$.

If the set of cyclic groups $\langle \varphi_{0}^{*}\rangle, \ldots, \langle \varphi_{n}^{*}\rangle$ coincides with the set of cyclic groups 
\[\langle \varphi_{0}\rangle, \ldots, \langle \varphi_{n}\rangle,\] then clearly $H_{0}=H$.

So, let us assume, from now on, that the above is not the case. 

\s
\noindent
\begin{claim}
The set of cyclic groups $\langle \varphi_{0}^{*}\rangle, \ldots, \langle \varphi_{n}^{*}\rangle$ is not disjoint with the set of cyclic groups $\langle \varphi_{0}\rangle, \ldots, \langle \varphi_{n}\rangle$. 
\end{claim}
\begin{proof}
Let us assume, by the contrary, that 
the set of cyclic groups $\langle \varphi_{0}^{*}\rangle, \ldots, \langle \varphi_{n}^{*}\rangle$ is disjoint with the set of cyclic groups $\langle \varphi_{0}\rangle, \ldots, \langle \varphi_{n}\rangle$. In this case, 
the group $H$ descends under the quotient map $\pi_{0}$, defined in eq. (\ref{pi0}),  to a group of M\"obius transformations that preserves the $n+1$ branch values $\infty$, $0$, $1$, $\lambda_{1}$,..., $\lambda_{n-2}$, and it is isomorphic to ${\mathbb Z}_{k}^{t}$, for  some $t \geq 1$.

It is known that the finite abelian subgroups of M\"obius transformations are either cyclic,  isomorphic to ${\mathbb Z}_{2}^{2}$ or isomorphic to 
$\mathbb{Z}_p^t$, where $p$ is the characteristic and $t \in \mathbb{N}$. The last case can not appear 
since $(k,p)=1$. 

\subsubsection*{Case 1}
If $k \geq 3$, then $t=1$ and $H \cap H_{0} \cong {\mathbb Z}_{k}^{n-1}$.
The cyclic group ${\mathbb Z}_{k}$ induced by $H$ is generated by a M\"obius transformation $T$ that permutes the $n+1$ branch values and fixes no one. In particular, $n+1=rk$, for some positive integer $r$. It follows (see \cite{GHL09}) that each lifting of $T$ (that is, the generators $\varphi_{0}^{*}, \ldots, \varphi_{n}^{*}$) is a linear transformation providing the same permutation (by conjugation action) of the generators $\varphi_{0},\ldots,\varphi_{n}$, in $r$ disjoint cycles of lenght $k$.  
Up to permutation of indices, we may assume that $\varphi_{0}^{*}$ permutes cyclically the elements of each of the sets $\{\varphi_{0}, \varphi_{1}, \ldots, \varphi_{k-1}\}$, $\{\varphi_{k}, \varphi_{k+1}, \ldots, \varphi_{2k-1}\}$,..., $\{\varphi_{(r-1)k}, \varphi_{(r-1)k+1}, \ldots, \varphi_{rk-1}\}$.
It follows that the maximal subgroup $Q$ of $H_{0}$ formed by those elements that commute with $\varphi_{0}^{*}$ is the one generated by the elements
$$\varphi_{0} \circ \varphi_{1}\circ \cdots \circ \varphi_{k-1}, \varphi_{k}\circ \varphi_{k+1}\circ \cdots\circ \varphi_{2k-1}, \ldots, \varphi_{(r-1)k}\circ \varphi_{(r-1)k+1}\circ \cdots\circ \varphi_{rk-1}.$$

Since the composition of all of the above elements equals the identity, $Q \cong {\mathbb Z}_{k}^{r-1}$.

Now, as $\varphi_{0}^{*}$ must commute with each element of $H \cap H_{0}$, the $n-1$ generators of it must be each one invariant under conjugation by $\varphi_{0}^{*}$. As  $H \cap H_{0}<Q$, we must have $n \leq r$, a contradiction.

\subsubsection*{Case 2} 
If $k=2$, then $t \in \{1,2\}$. If $t=1$, then we may proceed as in the above case to get a contradiction. If $t=2$, then $H \cap H_{0} \cong {\mathbb Z}_{2}^{n-2}$ and the group $H$ induces a group of M\"obius transformation isomorphic to ${\mathbb Z}_{2}^{2}$ that permutes the $n+1$ branch values and none of them is fixed by a non-trivial element. It follows that $n+1=4r$, for some positive integer $r$. 

In this case, after a permutation of the indices, we may assume that ${\mathbb Z}_{2}^{2}$ is generated by the induced elements of $\varphi_{0}^{*}$ and $\varphi_{1}^{*}$. It follows that $\varphi_{i}^{*}$ ($i=0,1$) permutes (by conjugation action) the generators $\varphi_{0},\ldots,\varphi_{n}$ in $2r$ disjoint cycles of lenght $2$ each one. Up to a permutation of indices, we may assume that $\varphi_{0}^{*}$ permutes cyclically the elements of each of the sets $\{\varphi_{0}, \varphi_{1}\}$, $\{\varphi_{2}, \varphi_{3}\}$,..., $\{\varphi_{n-1}, \varphi_{n}\}$. It follows that the maximal subgroup $Q$ of $H_{0}$ formed by those elements that commute with $\varphi_{0}^{*}$ is the one generated by the elements
$$\varphi_{0} \circ \varphi_{1}, \varphi_{2}\circ \varphi_{3},  \ldots, \varphi_{n-1}\circ \varphi_{n},$$
that is, $Q \cong {\mathbb Z}_{2}^{2r-1}$. Since the subgroup of $H_{0}$ formed by those elements that commute with $\varphi_{0}^{*}$ and with $\varphi_{1}^{*}$ is a subgroup of $Q$, we must that that $H \cap H_{0}<Q$, that is, $n-2 \leq 2r-1$. This obligates to have $r=1$, in particular, that $n=3$, a contradiction to the assumption that $(k-1)(n-1)>2$.
\end{proof}

\s

As a consequence of the above,  
the set of cyclic groups $\langle \varphi_{0}^{*}\rangle, \ldots, \langle \varphi_{n}^{*}\rangle$ is not disjoint with the set of cyclic groups $\langle \varphi_{0}\rangle, \ldots, \langle \varphi_{n}\rangle$. We may assume, up to permutation of the indices, that $\langle \varphi_{0} \rangle=\langle \varphi_{0}^{*}\rangle$.  The underlying Riemann surface $R$ of the quotient orbifold $(C(\lambda_{1},\ldots,\lambda_{n-2})/\langle \varphi_{0}\rangle$ is a generalized Fermat curve of type $(k,n-1)$ admiting two different generalized Fermat groups of type $(k,n-1)$; these being $H/\langle \varphi_{0}^{*} \rangle$ and the other being $H_{0}/\langle \varphi_{0} \rangle$. 

In the case that $K={\mathbb C}$ we have the following. For $k=2$ we have already proved the uniqueness (so normality) for $n=4,5$ in \cite{CGHR08} and for $k \geq 3$, the uniqueness was obtained for $n=3$ \cite{FGHL13}. In this way, the above procedure asserts, by induction on $n$, the desired result in the zero characteristic situation.

The situation for general $p>0$ can be done as follows. First, we know the uniqueness for $k \geq 4$ and $n=2$ (as a consequence of the results in \cite{Tze95} and \cite{Leo96}); so again, by the induction process we are done for $k \geq 4$. The case $k=2$ is ruled out because $1=k-1=p^{0}$ and we are assuming that $k-1$ is not be a power of $p$. In the case $k=3$, we only need to check uniqueness for $n=3$. 

\subsection*{The case $(k,n)=(3,3)$} In this case, our hypothesis are that $p \neq 2,3$. 
Lemma \ref{lema2} asserts that $H_{0} \cong {\mathbb Z}_{3}^{3}$ is a normal subgroup of 
$L$ and Lemma \ref{normL} asserts that $L$ is a normal subgroup of 
${\rm Aut}(F_{3,3})$. $W<L$ be the $3$-Sylow subgroup of $L$
containing $H_{0}$. If $W=H_{0}$, then the conditions of normality asserts the uniqueness. Let us now assume that $H_{0} \neq W$.
In this case, $W/H_{0}$ produces a $3$-subgroup $G<{\rm PGL}_{2}(K)$ keeping invariant the set 
$\{\infty,0,1,\lambda_{1}\}$. The only possibility is to have $G \cong {\mathbb Z}_{3}$. 
Up to a transformation in ${\rm PGL}_{2}(K)$, we may assume that the generator $T$ of $G$ satisfies that 
$T(\infty)=0$, $T(0)=1$, $T(1)=\infty$ and $T(\lambda_{1})=\lambda_{1}$. So, $T(x)=1/(1-x)$ and 
$\lambda_{1}^{2}-\lambda_{1}+1=0$. In this case, the collection $\{\infty,0,1,\lambda_{1}\}$ is also invariant under 
the involutions $A(x)=\lambda_{1}/x$ and $B(x)=(x-\lambda_{1})/(x-1)$. The group generated by $A$ and $B$ is
${\mathbb Z}_{2}^{2}$. In fact, the group $U$ generated by $A$ and $T$ is the alternating group ${\mathcal A}_{4}$ and it contains $B$. There are not more elements of ${\rm PGL}_{2}(K)-U$ 
keeping invariant the set $\{\infty,0,1,\lambda_{1}\}$; so $L/H_{0}=U \cong {\mathcal A}_{4}$. 
This ensures that $|L|=12 \times 3^3$ and also that $H_{0}$ is unique inside $L$ 
(see \cite[Cor. 6]{FGHL13}).
\end{proof}

\s
\noindent
\begin{lemma}\label{lema4}
Under the above assumptions, $N(H_{0})={\rm Aut}(F_{k,n})$, in particular, that ${\rm Aut}(F_{k,n})<{\rm PGL}_{n+1}(K)$.
\end{lemma}
\begin{proof}
Let 
$\tau \in {\rm Aut}(F_{k,n})$. Since $L$ is a normal subgroup of the group  ${\rm Aut}(F_{k,n})$ (see Lemma \ref{normL}), then 
$H=\tau H_{0} \tau^{-1}$ is a subgroup of $L$; again a generalized Fermat group of type $(k,n)$. Since $H_{0}$ is the unique generalized Fermat group of type $(k,n)$ inside $L$ (see Lemma \ref{lema3}), we must have that $H=H_{0}$.
\end{proof}

\s

\subsection{Conclusion of the proof of Theorem \ref{unicidad-p}}

Under our assumptions, $H_{0}$ is the unique generalized Fermat group of type $(k,n)$ of $F_{k,n}$.
In fact, since $L=N(H_{0})$ (Lemma \ref{lema2}), $N(H_{0})={\rm Aut}(F_{k,n})$ (Lemma \ref{lema4}) and $H_{0}$ is the unique generalized Fermat group of type $(k,n)$ inside $L$ (Lemma \ref{lema3}), the desired uniqueness result follows. 

The uniqueness ensures that ${\rm Aut}(F_{k,n})=N(H_{0})$. In \cite{GHL09} we obtain that $N(H_{0})$ is a subgroup of ${\rm PGL}_{n+1}(K)$. Now Lemma \ref{lemalineal1} provides the last part of our theorem.

\section{Proof of Theorem \ref{teolineal}}\label{Sec:4}
Before to provide the proof of Theorem \ref{teolineal} lest provide some general facts on linear automorphisms in algebraic varieties.

\s
\noindent
\begin{proposition}
Consider a complete intersection $Y\subset \mathbb{P}^s$ of projective hypersurfaces $Y_i$ of degree $d_i$ for $i=1,\ldots,r$. 
The canonical sheaf $\omega_Y$ is given by  
\[
\omega_Y=\mathcal{O}_Y\left( 
\sum_{i=1}^r d_i -s-1
\right).
\]
\end{proposition}
\begin{proof}
\cite[exer. 8.4 p. 188]{Har77}
\end{proof}

\s

The curve $F_{k,n}$ is given as complete intersection of $n-1$ hypersurfaces of degree $k$. Therefore, we have the following 

\s
\noindent
\begin{corollary}
The canonical sheaf on the curves $F_{k,n}$ is given by 
\[
\omega_{F_{k,n}}=\mathcal{O}_{F_{k,n}}\big( (n-1)k-n-1 \big)=
\mathcal{O}_{F_{k,n}}\big( (n-1)(k-1)-2 \big).
\]
\end{corollary}

\s

Of course this is  compatible with the genus computation given in \cref{genero} since 
the degree of $\mathcal{O}_{F_{k,n}}(1)$ is $k^{n-1}$.  

\s
\noindent
\begin{proposition}
Let $i:X \hookrightarrow \mathbb{P}^s$ be a closed projective subvariety, such that
the map 
\[
H^0\big(\mathbb{P}^s,\mathcal{O}_{\mathbb{P}^s}(1)\big) 
\stackrel{i^*}{\longrightarrow} 
H^0\big(X,\mathcal{O}_X(1)\big)
\]
is an isomorphism. Every automorphism of $X$ preserving $\mathcal{O}_X(1)$
can be extended to an automorphism of the ambient projective space, i.e. it is an element in $\mathrm{PGL}_{s+1}(K)$.  
\end{proposition}
\begin{proof}
\cite[prop. 2.1]{Kon02}
\end{proof}

\s

We may try to prove that every automorphism is linear in the following way.
Every automorphism $\sigma$ of the curve $F_{k,n}$ should preserve the canonical sheaf so it should preserve $\mathcal{O}_{F_{k,n}}\big( (n-1)(k-1)-2 \big)$. Does it preserve $\mathcal{O}_{F_{k,n}}(1)$? This is certainly  true if $\mathrm{Pic}(F_{k,n})$ has no torsion
and it is the general way how one proves linearity in higher dimensional varieties. Unfortunately curves have torsion in their Picard group. 

\s
\subsection{Proof of linearity part of Theorem \ref{teolineal}}\label{Sec:parte(1)}
Let $D=\mathcal{O}_{F_{k,n}}(1)$. For every automorphism $\sigma \in {\rm Aut}(F_{k,n})$ we consider the difference $T_\sigma:=\sigma(D)-D$. It is a divisor of degree $0$, and the divisor 
$\big( (n-1)(k-1)-2 \big) T_\sigma$ is principal. Hence $T_\sigma$ is a 
$\big( (n-1)(k-1)-2 \big)$-torsion point in the Jacobian of the curve $F_{k,n}$.
The automorphism is linear if and only if $T_\sigma$ is zero. 

\s
\noindent
\begin{lemma}
The map $\sigma \mapsto T_\sigma$ is a derivation, i.e.
\[
T_{\sigma \tau}=\sigma T_\tau + T_\sigma.
\]
\end{lemma}
\begin{proof}
Observe that 
\[
T_{\sigma \tau}=\sigma \tau (D)-D=\sigma \tau(D) -\sigma(D) +\sigma(D)-D=\sigma(T_\tau)+T_\sigma.
\]
\end{proof}

\s
\noindent
\begin{lemma}
The torsion points $T_\sigma$ are $H_{0}$-invariant.
\end{lemma}
\begin{proof}
Using Lemma \ref{normL} we find an $\ell \in L$ such that $h \sigma=\sigma \ell$. 
For all linear automorphisms $\ell$ and in particular for $\ell \in H_{0}$ we have $T_\ell=0$. We now use the 
derivation rules:
\[
T_{h\sigma}=h T_{\sigma}+ T_h=h T_{\sigma}
\]
and 
\[
T_{\sigma \ell}= \sigma T_\ell + T_{\sigma}=T_{\sigma}.
\]
The desired result follows, since $T_{h \sigma}=T_{\sigma \ell}$. 
\end{proof}

\s

Consider the natural map $\pi:F_{k,n} \rightarrow F_{k,n}/H_{0}\cong\mathbb{P}^1$. 
We have two maps induced on the Jacobians, namely
\[
\pi_*: \mathrm{Jac}(F_{k,n}) \rightarrow \mathrm{Jac}(F_{k,n}/H_{0})
\]
\[
\sum n_P P \mapsto \sum n_P \pi(P), 
\]
and 
\[
\pi^*:\mathrm{Jac}(F_{k,n}/H_{0})  \rightarrow \mathrm{Jac}(F_{k,n}) 
\]
\[
\sum n_Q Q \mapsto \sum n_Q \sum_{P \in \pi^{-1}(Q)} e(P/Q)  P, 
\]
where $\sum n_P P$ (resp. $\sum n_Q Q$) is a divisor of degree $0$ in $F_{k,n}$ (resp. $\mathbb{P}^1$) 
and $e(P/Q)$ denotes the ramification index of a point $P$ lying above $Q$.

Observe that the map  $\pi^*\circ \pi_*:\mathrm{Jac}(F_{k,n}) \rightarrow \mathrm{Jac}(F_{k,n})$  
 is given by sending a point $P\in \mathrm{Jac}(F_{k,n})$
 to $\sum_{h \in H_{0}} P$. On the other hand side $\pi_* \circ \pi^*$ is the zero map 
 since the Jacobian of the projective line is trivial. 
 
This means that on the $H_{0}$-invariant points $P_\sigma$, multiplication by $|H_{0}|=k^n$ is zero. 
Since $T_\sigma$ is an $\big( (n-1)(k-1)-2 \big)$-torsion point, if $(k,n+1)=1$, then 
$T_\sigma$ is zero and $\sigma$ is linear. 

\s
\subsection{Proof of second part of Theorem \ref{teolineal}}
Under the extra assumption that $(k,n+1)=1$, we have seen in Section \ref{Sec:parte(1)} that $L={\rm Aut}(F_{k,n})$. Now
Proposition \ref{powerp} states the last part of our theorem.

\s\s\s

{\bf Aknowledgment:} The authors would like to  thank  Gabino Gonz\'alez-Diez for his  remarks. 

\s


\end{document}